\documentclass[12pt]{iopart}
\usepackage{graphicx}
\graphicspath{{figures/}}
\usepackage{comment}
\usepackage{iopams} 
\usepackage{nicefrac}
\usepackage{amsthm}
\usepackage[section]{placeins}
\newtheorem{thm}{Theorem}[section]
\newtheorem{prop}[thm]{Proposition}
\newtheorem{cor}[thm]{Corollary}
\newtheorem{rem}[thm]{Remark}
\newtheorem{ass}[thm]{Assumption}
\newcommand{\norm}[1]{\left \|  #1 \right \|}
\newcommand{\abs}[1]{\left|  #1 \right|}
\newcommand{\argmin}{\mathop{\textnormal{argmin}}}
\newcommand{\real}[1]{\textnormal{Re} \left( #1 \right) }
\newcommand{\HS}{\mathop{\mathrm{HS}}}
\newcommand{\Cov}{\mathop{\mathrm{Cov}}}
\begin{document}

\title{Uniqueness of an inverse source problem in experimental aeroacoustics}

\author{Thorsten Hohage$^{1,2}$, Hans-Georg Raumer$^3$ and Carsten Spehr$^3$}
\address{$^1$Institute for Numerical and Applied Mathematics, University of G{\"o}ttingen, Germany}
\address{$^2$Max Planck Institute for Solar System Research, G{\"o}ttingen, Germany}
\address{$^3$Institute of Aerodynamics and Flow Technology, German Aerospace Center (DLR) G{\"o}ttingen, Germany}

\vspace{10pt}
\begin{indented}
\item[]\today
\end{indented}

\begin{abstract}
This paper is concerned with the mathematical analysis of experimental methods for the estimation of the power of an uncorrelated, extended aeroacoustic source from measurements of correlations of pressure fluctuations. We formulate a continuous, infinite dimensional model describing these experimental techniques based on the convected Helmholtz equation in $\mathbb{R}^3$ or $\mathbb{R}^2$.  
As a main result we prove that an unknown, compactly supported source power function 
is uniquely determined by idealized, noise-free correlation measurements. 
Our framework further allows for a precise characterization of state-of-the-art source reconstruction methods and their interrelations.
\end{abstract}
\vspace{2pc}
\section{Introduction} 
In experimental aeroacoustics one measures acoustic, randomly generated signals and aims at reconstructing the power of sources. In this paper we consider time-harmonic sound propagation in homogeneous flow fields which may be considered as simplified models of wind tunnel experiments. 
In that case acoustic pressure fluctuations may be caused by fluid-structure interactions or local turbulent structures inside the flow field which are then propagated towards a measurement array in the homogeneous main flow.  

The experimental investigation of aeroacoustic sound sources began in the 1970-s. Back then, the standard measurement device was an elliptic mirror \cite{Grosche1975}. However, the application of microphone arrays soon found its way into the field. Active and passive microphone array methods are widely applied for the localization of sources of wave fields or the  imaging of the propagation medium \cite{Moscoso2018, Garnier2016}. The field of applications covers many branches of physics and engineering for example radar (see, e.g., \cite{Haykin1993}) or geophysics (see, e.g., \cite{Bleistein2001}).  The first fundamental work on microphone array imaging methods in aeroacoustics, published in 1976 by Billingsley \& Kinns \cite{Billingsley1976}, deals with aeroacoustic sound sources of a turbulent jet. Since then, the data processing and evaluation techniques for microphone array data were constantly developed further and microphone arrays are nowadays the standard experimental measurement devices for aeroacoustic experiments. Some of the most common applications of microphone arrays for (aero-)acoustic purposes are aircraft measurements (fly-over \cite{Howell1986} or inside a wind tunnel \cite{Underbrink2002}), jet noise \cite{Suzuki2010}, wheel/rail noise of trains \cite{Barsikow1987} and wind turbines \cite{Oerlemans2007}. 
There exists also a close connection to helioseismic holography, 
which back-propagates correlations of acoustic waves observed on the Sun's near-side to its interior or far-side to study the structure and dynamics therein
(see Lindsey \& Brown \cite{LB:00,LB:00a}, Gizon et al.\  \cite{GFYBB:18}, and Section \ref{sec:beamforming}). 
This article considers only free field sound propagation since many microphone array methods rely on this assumption. Sound propagation models based on numerical simulation allow more complex geometrical setups. A source localization method using Finite Element simulation on two- and three-dimensional geometries is presented by Kaltenbacher, Kaltenbacher \& Gombots in \cite{Kaltenbacher2018, Gombots2018}.

In wind tunnel experiments and many other applications discussed above it is common practice to compute the correlation matrix of the measurement array in a preprocessing step and reconstruct source powers from these correlation data. The main aim of this work is to establish a uniqueness result for the inverse problem of reconstructing bounded and compactly supported source power functions from correlation data in a continuous setting. The analysis is carried out for arbitrary constant, subsonic convection and therefore, as a special case, also covers linear acoustics in a medium at rest. 

Let us distinguish our problem from two related inverse source problems with rather different properties. We are concerned with spatially extended sources, in contrast to the acoustic localization of a small number of point sources. Such \emph{source localization problems} occur for example in speaker localization or speech enhancement for hearing aid devices (see, e.g., \cite{Alexandridis2015,Farmani2017, Ganguly2017, Long2019, Tachioka2014,WTdB:09}). Methods for this scenario seek to localize a small number of sources inside a reverberant room. Often, the direction of arrival (DOA) of the acoustic sources is of special interest, since it allows to suppress signals from other directions (speech enhancement). One of the main difficulties for this class of problems is that the source signal is superposed with disturbance signals due to wall reflections. 

Moreover, whereas we consider the  identification of fully uncorrelated sources, for the \emph{identification of deterministic  (i.e.\  completely correlated) sources} it is well known and easy to see that such sources are not uniquely determined from distant measurements of acoustic waves since there exist so-called non-radiating sources. Nevertheless, significant  progress has been achieved recently by Griesmaier \& Sylvester in retrieving partial information from such sources, e.g.\ stably splitting well-separated sources, see \cite{GS:16,GS:17,GS:17b}. Note that for deterministic sources, data consist of deterministic wave fields, whereas for random sources one may consider correlation data. The fact that correlation functions depend on more independent variables than the corresponding wave fields is a first formal indication that uniqueness results for uncorrelated sources may be possible.  

Since we consider spatially extended sources, a continuous source representation is natural for our scenario. Furthermore, in wind tunnel experiments sources may be  considered as non-deterministic  since the sound field is generated and measured inside a flow field with turbulent structures. As already mentioned above, one often uses the correlations between the microphone signals as input for the reconstruction process. One of the basic reconstruction techniques, based on microphone correlation measurements are Beamforming methods \cite{Allen2002}. To improve the spatial resolution of Beamforming outputs, post processing methods like DAMAS \cite{Brooks2006} and Clean-SC \cite{Sijtsma2007} have been proposed. Covariance Matrix Fitting \cite{Blacodon2004, Yardibi2008} is an inverse method that reconstructs source powers directly from the measured correlation matrix. We will review these methods from a continuous perspective given by the setting of our uniqueness result. 

The plan of the remainder of this paper is as follows: 
The forward problem  for time harmonic sound propagation of uncorrelated sources is introduced in Section \ref{sec:fwd_prob} before presenting  our main uniqueness result in Section \ref{sec:theory}. In Section \ref{sec:incorporation} we generalize commonly used discrete reconstruction methods to our continuous framework, study their interrelations and compare their performance for an experimental data example. 
Finally, we end this paper with some conclusions.  


\section{The forward problem} \label{sec:fwd_prob}
We consider as geometrical setup (see also Figure \ref{fig:exp_sketch})
\begin{itemize}
\item $d \in \lbrace 2,3 \rbrace$
\item a bounded, open domain $\Omega \subset \mathbb{R}^d$ with $\mathbf{0} \in \Omega$ (source region) \\such that $\mathbb{R}^d\backslash  \overline{\Omega}$ is connected and
\item a bounded, open domain $\mathbb{M} \subset \mathbb{R}^d \backslash \overline{\Omega}$  (measurement region).
\end{itemize}
\begin{figure}[ht!]
\centering
\includegraphics[width = .5\textwidth]{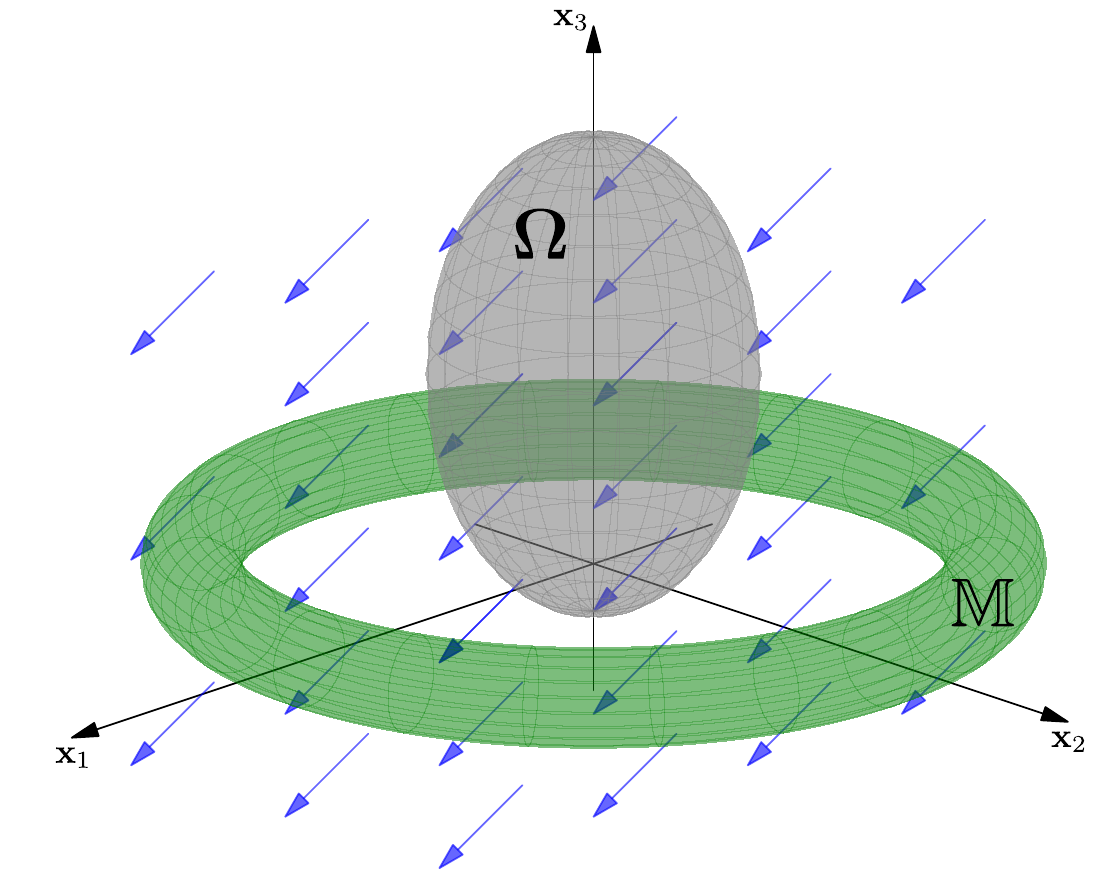}
\caption{Sketch of the geometrical setup. The flow field is indicated in blue.} \label{fig:exp_sketch}
\end{figure}
The sound propagation model inside a homogeneous flow is given by the convected Helmholtz equation for a subsonic, constant flow field $\mathbf{u} \in \mathbb{R}^d$. Let $c$ denote the speed of sound and $\mathbf{m} = \frac{1}{c} \mathbf{u}$ the Mach vector. We consider a subsonic regime, i.e.\ we assume that $|\mathbf{m}| < 1$. For the  sign convection $e^{-i \omega t}$ for the time factor, the convected Helmholtz equation for a function $p$ and a source term $Q$ reads as 
\begin{equation} \label{eq:conv_helmeq}
(k + i \mathbf{m} \cdot \nabla)^2p + \Delta p =  -Q \ .
\end{equation}
The free field Green's function for equation \eref{eq:conv_helmeq} in three dimensions is (cf. \cite[Appendix A]{Mosher1984})
\begin{equation} 
g(\mathbf{x},\mathbf{y}) =  \frac{\textnormal{exp} \left(\frac{ik}{\beta^2} \left( -(\mathbf{x}-\mathbf{y})\cdot \mathbf{m} + |\mathbf{x}-\mathbf{y}|_{\mathbf{m}} \right) \right) }{4 \pi |\mathbf{x}-\mathbf{y}|_{\mathbf{m}} } \ ,\label{eq:gf3d_convhelmeq}
\end{equation}
with the Mach scaled distance 
\begin{equation}
|\mathbf{x}-\mathbf{y}|_{\mathbf{m}} = \sqrt{\left((\mathbf{x}-\mathbf{y}) \cdot \mathbf{m} \right)^2 + \beta^2 |\mathbf{x}-\mathbf{y}|^2} \label{eq:machdist}
\end{equation}
and $\beta^2 = 1- |\mathbf{m}|^2$. Note that  $|\cdot|_{\mathbf{m}}$ \eref{eq:machdist} is a norm on $\mathbb{R}^d$ which is induced by a scalar product. Recall that the free Green's function for $\mathbf{m}=0$ in two dimensions is
\begin{equation*}
    g(\mathbf{x}, \mathbf{y}) = \frac{i}{4} H^{(1)}_0(k \abs{\mathbf{x}-\mathbf{y}} ) 
\end{equation*}
where $H^{(1)}_0$ denotes the Hankel function of the first kind of order $0$ (see \cite[p.74]{Colton2013}). Straightforward computations show that the free Green's function for general $\mathbf{m}$ with $\abs{\mathbf{m}}<1$ is given by
\begin{equation}
g(\mathbf{x}, \mathbf{y}) = \frac{i}{4 \beta} \textnormal{exp} \left( -\frac{ik}{\beta^2} (\mathbf{x} - \mathbf{y}) \cdot \mathbf{m} \right) H^{(1)}_0 \left( \frac{k}{\beta^2} \abs{\mathbf{x}-\mathbf{y}}_{\mathbf{m}}  \right) \ . \label{eq:gf2d_convhelmeq}
\end{equation} 
In a real experimental setup and evaluation process, the data and the reconstructed quantities are finite. We consider $M$ microphones in the measurement region $\mathbb{M}$ at positions $\left \lbrace \mathbf{x}_m \right \rbrace_{m=1}^M$. We also discretize the source region $\Omega$ by $N$ focus points $\left \lbrace \mathbf{y}_n \right \rbrace_{n=1}$ and corresponding disjoint sets 
$\Omega_n$ such that $\mathbf{y}_n\in\Omega_n$ and $\bigcup_{n=1}\Omega_n = \Omega$. E.g., the sets $\Omega_n$ may be chosen as Voronoi cells. The random source term $Q$ is now approximated by a sum of scaled approximate delta distributions with random complex amplitudes $\Pi_n$, e.g.\ a piecewise constant function 
\begin{equation}\label{eq:sourcerep}
\underline{Q}(\mathbf{y}) = \sum_{n=1}^N\Pi_n \phi_{n}(\mathbf{y}) \ ,
\end{equation}
where $\phi_n =1_{\Omega_n}|\Omega_n|^{-1/2}$ is an $L^2$-normalized indicator function of $\Omega_n$. 
A standard assumption is that the complex amplitudes $\Pi_n$ have zero mean and are mutually uncorrelated:
\begin{ass}[Uncorrelated sources]\label{as:uncor} \ \\
The random source amplitude vector $\mathbf{\underline{\Pi}}= \left(\Pi_1, \dots, \Pi_N  \right)^{\top}$ satisfies 
\[
\mathbb{E}\left( \mathbf{\underline{\Pi}} \right) = \mathbf{0} \qquad \mbox{and}\qquad 
\textnormal{Cov}\left( \mathbf{\underline{\Pi}} \right) = \underline{M_q} 
\]
with the source powers $\mathbb{E} \left[  \abs{\Pi_n}^2 \right]  
= q(\mathbf{y}_n)$ for some continuous source power function $q:\overline{\Omega}\to [0,\infty)$ and  $\underline{M_q} = \textnormal{diag} 
\left(q(\mathbf{y}_1), \dots ,  q(\mathbf{y}_N)  \right)$.
\end{ass}
\paragraph{Notation:}Mathematical objects that posses a discrete and an infinite dimensional version will be indicated by the same symbol for the entire article. For distinction the discrete versions will be marked by an underscore.

At any point $\mathbf{x}\in\mathbb{R}^d$, the source term $\underline{Q}$ generates the pressure signal 
\begin{equation}
    \underline{p}(\mathbf{x}) := \int_{\Omega} g(\mathbf{x}, \mathbf{y}) \underline{Q}(\mathbf{y}) d \mathbf{y} \ . \label{eq:signalprop}
\end{equation}
Note that $\mathbb{E}[ \underline{p}(\mathbf{x})]=0$ under 
Assumption \ref{as:uncor}. 
For the microphone signal vector $\underline{\mathbf{p}}=\left( \underline{p}(\mathbf{x}_1), \dots , \underline{p}(\mathbf{x}_M)  \right)^{\top}$ equation \eref{eq:signalprop} yields
\begin{equation*}
    \underline{\mathbf{p}} = \underline{\mathcal{G}} \mathbf{\underline{\Pi}} \  ,
\end{equation*}
where the propagation matrix $ \underline{\mathcal{G}} \in \mathbb{C}^{M \times N} \ $ is defined by
\begin{equation}
\underline{\mathcal{G}}_{mn} =  
\int_{\Omega_n^N}g(\mathbf{x}_m, \mathbf{y})\,d\mathbf{y} |\Omega_n^{N}|^{-1/2} \approx
g(\mathbf{x}_m, \mathbf{y}_n)|\Omega_n|^{1/2} 
\ . \label{eq:prop_mat}
\end{equation}
In aeroacoustic array measurements, one estimates the covariance matrix (also \textit{cross correlation matrix} or \textit{cross-spectral matrix}) of the microphone signal vector $\underline{\mathbf{p}}$, where the estimation process is usually carried out by Welch's method \cite{Welch1967}. Assumption \ref{as:uncor} yields the following representation for the estimated cross correlation matrix $\underline{C}^{\mathrm{obs}}$
\begin{equation} 
\fl
\underline{C}^{\mathrm{obs}} 
\approx \Cov \left(\underline{\mathbf{p}} \right) = 
\Cov \left( \underline{\mathcal{G}} \mathbf{\underline{\Pi}}  \right) = \underline{\mathcal{G}} \underline{M_q}  \underline{\mathcal{G}}^{*}  =: \underline{\mathcal{C}} (\underline{q}) \ . \label{eq:fwd_discrete}
\end{equation}
If we let $N$ tend to infinity such that the fill distance in $\Omega$ tends to $0$, then under mild assumptions on the choice of $\Omega_n$ (e.g.\ for Voronoi cells) and the distribution of $\mathbf{\underline{\Pi}}$ (e.g.\ for Gaussian distributions) the 
random functions $\underline{Q}$ tend to a Gaussian random process $Q$ on $\Omega$ with covariance operator 
\begin{equation*}
(M_q v)(\mathbf{y}) := q(\mathbf{y})v(\mathbf{y}) 
\end{equation*}
in the sense that $\int_{\Omega}\psi \underline{Q}\varphi\,d\mathbf{y}$ 
converges in probability to $\langle \psi, Q\varphi \rangle$ 
for any $\psi,\varphi\in L^2(\Omega)$ and $\mathbb{E}[\langle \psi, Q\varphi \rangle] = 
\langle \psi, M_q\varphi \rangle$. 
(For $q\equiv 1$, the process $Q$ is Gaussian white noise.) Moreover, the limiting 
pressure signal $p(\mathbf{x}):=\langle Q, g(\mathbf{x},\cdot)\rangle = (\mathcal{G}Q)(\mathbf{x})$ 
can be written in terms of the 
volume potential operator $\mathcal{G}: L^2(\Omega) \rightarrow L^2(\mathbb{M})$
\begin{equation} \label{eq:volpot_op}
(\mathcal{G}v)(\mathbf{x}) := \int_{\Omega} g(\mathbf{x},\mathbf{y}) v(\mathbf{y}) d\mathbf{y}. 
\end{equation}
Note that $\mathbb{E}[p(\mathbf{x})]=0$ and that 
$p$ has the covariance function  
\begin{equation*}
c_q(\mathbf{x}_1, \mathbf{x}_2) := \Cov(p(\mathbf{x}_1),p(\mathbf{x}_2))\
= \int_{\Omega} g(\mathbf{x}_1, \mathbf{y}) q(\mathbf{y}) \overline{ g(\mathbf{x}_2, \mathbf{y})} d\mathbf{y} \ ,
\end{equation*}
which under the assumptions above is the limit of 
$\Cov(\underline{p}(\mathbf{x}_1),\underline{p}(\mathbf{x}_2))$ as $N\to\infty$. This implies that the integral operator $\mathcal{C}(q):L^2(\mathbb{M})\to L^2(\mathbb{M})$, 
\begin{equation*}
\left( \mathcal{C}(q) \varphi \right)(\mathbf{x}_1) := \int_{\mathbb{M}} c_q(\mathbf{x}_1, \mathbf{x}_2) \varphi(\mathbf{x}_2) d\mathbf{x}_2 \ 
\end{equation*}
(the covariance operator of the acoustic pressure signal $p$) can be written as 
\begin{equation}\label{eq:fwd_infinite}
\mathcal{C}(q) = 
\mathcal{G} M_q \mathcal{G}^* \ .
\end{equation}

We will consider the forward operator $\mathcal{C}$ as an operator-valued linear mapping from $L^2\left(  \Omega \right)$ to the space $\HS\left( L^2(\mathbb{M} )\right)$ of Hilbert-Schmidt operators on $L^2(\mathbb{M})$
\begin{equation*}
\mathcal{C}: \ L^2(\Omega) \rightarrow \HS\left( L^2(\mathbb{M}) \right), \ \ \ \ q \mapsto \mathcal{C}(q) \ .
\end{equation*}
Recall that a compact operator 
$K:X\to X$ in a Hilbert space $X$ is called a Hilbert-Schmidt operator if the eigenvalues of $K^*K$ are summable. Equipped with the inner product 
$\langle K_1,K_2\rangle_{\HS}:=\mathrm{tr}\,(K_2^*K_1)$ the space $\HS(X)$ is a Hilbert space \cite[Theorem VI.22 (c)]{Reed1972}. As for integral operators, the Hilbert-Schmidt norm is given by the $L^2$-norm of the kernel function (see \cite[Theorem VI.23]{Reed1972}), we have  
\begin{equation}\label{eq:HSnormCq}
\|\mathcal{C}(q)\|_{\HS} = \|c_q\|_{L^2(\mathbb{M}\times \mathbb{M})}.
\end{equation}
As $c_q$ is the true input data of the inverse problem, this shows that $\HS(L^2(\mathbb{M}))$ is a natural choice for the image space. 

\begin{prop}[Mapping property of the forward operator] \label{prop:HS_cov}\ \\ For any $q \in L^2(\Omega)$ the operator $\mathcal{C}(q)$ belongs to $\HS\left( L^2(\mathbb{M}) \right)$. 
\end{prop}
\begin{proof}
See \ref{sec:appendix}.
\end{proof}


\section{Uniqueness result for the inverse problem}\label{sec:theory}
The main goal of this section is to prove a uniqueness result for the operator equation 
\begin{equation} \label{eq:opeq_uniqueness}
\mathcal{C}(q) =  C  
\end{equation}
for $C \in \textnormal{ran}\left( \mathcal{C} \right)$ and a bounded source power function $q \in L^{\infty}\left(\Omega \right) $.
The proof relies on several auxiliary statements which will be presented beforehand. For a constant and subsonic flow field $\mathbf{u}$, the Lorentz transformation (a.k.a. Prandtl-Glauert transformation in fluid dynamics) relates the Helmholtz equation to the convected Helmholtz equation. Therefore many well-known results for the Helmholtz equation are easily transferred to the case with convection.
\begin{prop}[Lorentz transformation] \label{prop:lorentztrafo}\ \\
Let $\mathbf{m} = (m_1, 0, \dots )^{\top}$ and $\mathbf{T}=\mathrm{diag}\left(\frac{1}{\beta},1, \dots \right)$. We consider an open domain $U\subset \mathbb{R}^d$ and functions $\phi_{0} \in C^2(U), Q \in C^0(U)$ with
\begin{equation*}
\left \lbrace  \frac{k^2}{\beta^2} + \Delta \right \rbrace \phi_0(\mathbf{x}) = -Q(\mathbf{x}) \quad \mathrm{on} \ U \ .
\end{equation*}
The transformed function
\begin{equation*}
\phi_{\mathbf{m}}(\mathbf{x}) = \mathrm{exp}\left( - \frac{\abs{\mathbf{m}}ik}{\beta^2} \mathbf{x}_1 \right)   \phi_0\left( \mathbf{T} \mathbf{x}  \right)
\end{equation*}
then solves the convected Helmholtz equation with a modified source term on a scaled domain. More precisely on $\mathbf{T}^{-1}(U)$ we have
\begin{equation}
\left \lbrace \left(k+i \mathbf{m} \cdot \nabla \right)^2  + \Delta \right \rbrace \phi_{\mathbf{m}}(\mathbf{x}) = -\mathrm{exp}\left( -\frac{\abs{\mathbf{m}}ik}{\beta^2} \mathbf{x}_1 \right)  Q\left( \mathbf{T}\mathbf{x}\right) \ . \label{eq:lorentzsol}
\end{equation}
\end{prop}
\begin{proof}
A more general result for the wave equation in time domain holds with the coordinate transform
\begin{eqnarray*}
\mathbf{\widehat{T}}
\left(\left(\!\begin{array}{c}
\mathbf{x}\\t
\end{array}\!\right)\right)
&:= & \left(\!\begin{array}{c}\mathbf{T} \mathbf{x} \\
 \beta t + \frac{\abs{\mathbf{m}}}{\beta c} \mathbf{x}_1 
 \end{array}\!\right)\ ,
\end{eqnarray*} 
see e.g. \cite{Chapman2000, Gregory2015}. It states that if 
$
\left[ \Delta^2 - \frac{1}{c^2} \frac{\partial^2}{t^2} \right] \psi_0(\mathbf{x}, t) = -F(\mathbf{x}, t)
$
then the pullback $\psi_{\mathbf{m}}:=\psi_0\circ\mathbf{\widehat{T}}$
satisfies
\begin{equation*}
\left[  \Delta - \frac{1}{c^2} \left(  \frac{\partial}{\partial t} + \mathbf{u}_1 \frac{\partial}{\partial \mathbf{x}_1} \right)^2   \right] \psi_{\mathbf{m}} = -F\circ\mathbf{\widehat{T}} \ .
\end{equation*}
The statement \eref{eq:lorentzsol} is then obtained from this time domain result for the special case of time harmonic functions $\psi_0(\mathbf{x}, t) = \phi_0(\mathbf{x}) e^{-i\frac{\omega}{\beta} t}$. 
\end{proof}
For the convected Helmholtz equation propagation directions will be elements of  the unit sphere with respect to the Mach-norm given by
\begin{equation*}
\mathbb{S}_{d-1}(\mathbf{m}) := \lbrace \mathbf{x} \in \mathbb{R}^d: \  |\mathbf{x}|_{\mathbf{m}} = 1 \rbrace \ .
\end{equation*}
The next proposition describes the asymptotic behavior of the Green's functions \eref{eq:gf3d_convhelmeq}/\eref{eq:gf2d_convhelmeq} at infinity. 
\begin{prop}[Asymptotic behavior of Green's function]\label{prop:as_green} \ \\Let $B = B_R(\mathbf{0})$ be a ball, with radius $R$ such that $\Omega \subset B$, then for $\mathbf{x} \in \mathbb{R}^d \backslash \overline{B}$ and $\mathbf{y} \in B$ the following asymptotic representation holds true:
\begin{equation}\label{eq:asymp_gf}
 \fl g(\mathbf{x},\mathbf{y}) = C(d)h(\mathbf{x}) \abs{\mathbf{x}}_{\mathbf{m}}^{-\frac{(d-1)}{2}}  \textnormal{exp}\left( \frac{ik}{\beta^2} \left( \mathbf{m} -\mathbf{A}\hat{\mathbf{x}} \right) \cdot \mathbf{y} \right) 
 +  \mathcal{O}\left( |\mathbf{x}|^{-\frac{(d+1)}{2}} \right)  \quad {  \textnormal{as} \  |\mathbf{x}|\rightarrow \infty   }   
\end{equation}
with the auxiliary quantities
\begin{eqnarray*}
C(2) := \frac{e^{i\frac{\pi}{4}}}{\sqrt{8 \pi k}}, \quad &
h(\mathbf{x}) :=\textnormal{exp} \left(  \frac{ik}{\beta^2} \left( |\mathbf{x}|_{\mathbf{m}} - \mathbf{x} \cdot \mathbf{m} \right) \right),  &\\
C(3) := \frac{1}{4 \pi}, \quad &
 \mathbf{A} := \mathbf{m} \mathbf{m}^{\top} + \beta^2 \mathbf{I}, &
\hat{\mathbf{x}} := \frac{\mathbf{x}}{|\mathbf{x}|_{\mathbf{m}}}.  
\end{eqnarray*}
The asymptotic formula \eref{eq:asymp_gf} holds uniformly for all $\mathbf{y} \in B$ and all directions $\hat{\mathbf{x}} \in \mathbb{S}_{d-1}(\mathbf{m})$.
\end{prop} 
\begin{proof}
For $\mathbf{m}=0$ this is a standard result, see \cite{Colton2013}. For $\abs{\mathbf{m}}>0$ assume w.l.o.g. that $\mathbf{m}=(m_1, 0,  \dots)^{\top}$. Let $g_0(\mathbf{x}, \mathbf{y}, k)$ (resp. $g_{\mathbf{m}}(\mathbf{x}, \mathbf{y}, k)$) denote the Green's function for the Helmholtz equation (resp. convected Helmholtz equation) at wavenumber $k$. In accordance with Proposition \ref{prop:lorentztrafo} we observe that 
\begin{eqnarray*}
g_{\mathbf{m}}(\mathbf{x},\mathbf{y}, k) &=& \frac{1}{\beta} \mathrm{exp}\left( -\frac{ik}{\beta^2} (\mathbf{x}-\mathbf{y}) \cdot \mathbf{m} \right) g_0\left(\mathbf{T}\mathbf{x}, \mathbf{T} \mathbf{y}, \frac{k}{\beta}  \right) \ .
\end{eqnarray*}
Using the asymptotic behavior of $g_0$ yields the claim.
\end{proof}

For the volume potential $w$ of a function $v \in L^2(\Omega)$
\begin{equation}\label{eq:volpot_def}
w(\mathbf{x}) = \int_{\Omega} v(\mathbf{y}) g(\mathbf{x},\mathbf{y}) d\mathbf{y} \ , 
\end{equation}
we immediately get a similar representation with the far field pattern 
\begin{equation}
w^{\infty}(\hat{\mathbf{x}}) = \int_{\Omega} \textnormal{exp}\left( \frac{ik}{\beta^2} \left( \mathbf{m} -\mathbf{A}\hat{\mathbf{x}} \right) \cdot \mathbf{y} \right) v(\mathbf{y}) d\mathbf{y} \ . \label{eq:farfield}
\end{equation}

\begin{cor}[Asymptotic representation of the volume potential]\label{cor:vol_pot} \ \\
With $w^\infty$ defined in \eref{eq:farfield} 
the volume potential $w$ in  \eref{eq:volpot_def} has the asymptotic behavior 
\begin{equation*}
w(\mathbf{x}) = C(d)h(\mathbf{x})\abs{\mathbf{x}}_{\mathbf{m}}^{-\frac{(d-1)}{2}} w^{\infty}(\hat{\mathbf{x}}) + \mathcal{O}\left( |\mathbf{x}|^{-\frac{(d+1)}{2}} \right)  \ 
\end{equation*}
\end{cor} 
\begin{proof}
Inserting the asymptotic representation \eref{eq:asymp_gf} from the last proposition into the definition of the volume potential yields the claim. 
\end{proof}

In order to characterize elements of the kernel of the volume potential operator $\mathcal{G}$ in \eref{eq:volpot_op}, we show that the volume potential $w$ in  \eref{eq:volpot_def} is real analytic outside of the source region. 
\begin{prop}[Analyticity of the volume potential]\label{prop:analyticity} \ \\The volume potential $w$ \eref{eq:volpot_def} is real analytic on $\mathbb{R}^d \backslash \overline{\Omega}$. 
\end{prop} 
\begin{proof}
For $\mathbf{m}=0$ this is again well-known since $w$ is a $C^2$ solution of the homogeneous Helmholtz equation on $U:=\mathbb{R}^d \backslash \overline{\Omega}$ and therefore analytic (see \cite[Theorem 2.2]{Colton2013} for $d=3$ and \cite[Theorem 3.2]{Cakoni2006} for $d=2$). For $\abs{\mathbf{m}}>0$ the statement is obtained using the Lorentz transformation. 
\end{proof}
Now we can identify a set of specific plane waves as a subset of the closure of the range of the adjoint volume potential operator $ \mathcal{G}^*$.
\begin{prop}[Plane waves] \label{prop:plane_waves} \ \\
For the plane wave functions
\begin{equation*}
u_{\mathbf{m}, \hat{\mathbf{x}} }(\mathbf{y}) :=   \textnormal{exp}\left( \frac{ik}{\beta^2} \left(\mathbf{A} \hat{\mathbf{x}} - \mathbf{m} \right) \cdot \mathbf{y} \right) \ ,
\end{equation*}
the following inclusion holds true
\begin{equation*}
\mathcal{W} := \left \lbrace u_{\mathbf{m}, \hat{\mathbf{x}} } \  : \ \ \hat{\mathbf{x}} \in \mathbb{S}_{d-1}(\mathbf{m}) \right \rbrace \subset \overline{ \textnormal{ran}\left( \mathcal{G}^* \right)} \ .
\end{equation*}
\end{prop}

\begin{proof}
We will show that $\mathcal{W} \subset \textnormal{ker}\left( \mathcal{G} \right)^{\perp}$ which is equivalent to the claim. Assume that $v \in \textnormal{ker}\left( \mathcal{G} \right)$, i.e.\ the volume potential $w(\mathbf{x}) = \int_{\Omega} g(\mathbf{x},\mathbf{y}) v(\mathbf{y}) d\mathbf{y} $ vanishes on $\mathbb{M}$. Due to Proposition \ref{prop:analyticity}, $w$ is analytic and since it vanishes on the open set $\mathbb{M}$ and $\mathbb{R}^d \backslash \overline{\Omega}$ is connected, it must vanish on all of $\mathbb{R}^d \backslash \overline{\Omega}$ by analytic continuation \cite[p. 65]{John1982}. The representation formula of Corollary \ref{cor:vol_pot}, 
\begin{equation*}
w(\mathbf{x}) = C(d)h(\mathbf{x}) \abs{\mathbf{x}}_{\mathbf{m}}^{-\frac{(d-1)}{2}} w^{\infty}(\hat{\mathbf{x}}) + \mathcal{O}\left( |\mathbf{x}|^{-\frac{(d+1)}{2}} \right)
\end{equation*} 
implies that the far field pattern $w^{\infty}$ must also vanish identically. 
Together with the definition of the far field pattern we obtain
\begin{equation*}
0 = w^{\infty}( \hat{\mathbf{x}}) = \int_{\Omega} \overline{ u_{\mathbf{m}, \hat{\mathbf{x}} }(\mathbf{y})} v(\mathbf{y}) d \mathbf{y} = \left \langle  v, u_{\mathbf{m}, \hat{\mathbf{x}} } \right \rangle_{L^2(\Omega)}\,,
\end{equation*}
which shows that $\mathcal{W} \subset \textnormal{ker}\left( \mathcal{G} \right)^{\perp}$.  
\end{proof}
Finally, we are in the position to prove the uniqueness statement for bounded sources. 
\begin{thm}[Uniqueness] \label{thm:uniqueness} \ \\If $q_1, q_2 \in L^ {\infty}\left(\Omega \right)$ such that $\mathcal{C}(q_1) = \mathcal{C}(q_2)$, then  \begin{equation*}
q_1 = q_2 \ .
\end{equation*}
 \end{thm} 
\begin{proof}
Due to linearity is suffices to show that $\mathcal{C}(q) \equiv 0$ implies $q=0$. So let
$q \in L^{\infty}\left( \Omega \right)$ such that $\mathcal{C}(q) \equiv 0$. Then 
\begin{equation*}
0 =  \left \langle \mathcal{G} M_q \mathcal{G}^* v_1, v_2  \right \rangle_{L^2(\mathbb{M})} = \left \langle  M_q \mathcal{G}^* v_1, \mathcal{G}^*v_2  \right \rangle_{L^2(\Omega)}
\end{equation*}
for all $v_1,v_2 \in L^2(\mathbb{M})$ and hence 
\begin{equation} \label{eq:range_prop}
\left \langle  q u_1, u_2  \right \rangle_{L^2(\Omega)} = 0 \ \ \ \ \mbox{for all } u_1,u_2 \ \in \ \textnormal{ran}\left(  \mathcal{G}^* \right) \ .
\end{equation}
Since $q \in L^{\infty}(\Omega)$, we can apply a density argument to show that property \eref{eq:range_prop} holds also for elements of $\overline{\textnormal{ran}\left(  \mathcal{G}^* \right) }$. By Proposition \ref{prop:plane_waves} we can choose for $u_1$ and $u_2$ plane waves of the form $u_{\mathbf{m}, \hat{\mathbf{x}}_1}, u_{\mathbf{m}, \hat{\mathbf{x}}_2}$ with $\hat{\mathbf{x}}_1, \hat{\mathbf{x}}_2 \in  \mathbb{S}_{d-1}(\mathbf{m}) $. Together with \eref{eq:range_prop} this implies
\begin{equation} \label{eq:plane_eq}
0 = \left \langle q u_{\mathbf{m}, \hat{\mathbf{x}}_1}, u_{\mathbf{m}, \hat{\mathbf{x}}_2} \right \rangle_{L^2(\Omega)} = 
\int_{\Omega} q(\mathbf{y}) \textnormal{exp} \left(  \frac{-ik}{\beta^2} \mathbf{A} (\hat{\mathbf{x}}_2-\hat{\mathbf{x}}_1) \cdot \mathbf{y} \right) d\mathbf{y}  \ .
\end{equation}
Note that 
\begin{equation*}
\lbrace \hat{\mathbf{x}}_2 - \hat{\mathbf{x}}_1 \ : \  \hat{\mathbf{x}}_1, \hat{\mathbf{x}}_2 \in \mathbb{S}_{d-1}(\mathbf{m}) \rbrace = \lbrace \mathbf{x}  \in \mathbb{R}^d: \ |\mathbf{x}|_{\mathbf{m}} \leq 2 \rbrace 
\end{equation*}
contains an open set $O$ with respect to $| \cdot |$ and the set $V = \frac{k}{\beta^2}\mathbf{A}(O)$ is also open
as $\mathbf{A}$ is a homeomorphism on $\mathbb{R}^d$. 
Extending $q$ by zero to the whole space (denoting the extension again by $q$) and using \eref{eq:plane_eq} we obtain
\begin{equation} \label{eq:q_fourier}
\hat{q}(\boldsymbol{\xi}) = \int_{\mathbb{R}^d} q(y) e^{-i\boldsymbol{\xi} \cdot \mathbf{y}} d \mathbf{y}  = 0 \ \ \ \ \ \textnormal{for} \  \boldsymbol{\xi} \in V \ ,
\end{equation}
i.e.\ the Fourier transform of $q$ vanishes on the open set $V$. Since $q$ has compact support, $\hat{q}$ is real analytic \cite[Section IX.3]{Reed1975}, and hence it must vanish everywhere by analytic continuation \cite[p. 65]{John1982}. Since the Fourier transform is injective, we obtain that $q=0$.
\end{proof}

\begin{rem}[Correlated sources] \ \\
One may ask if it is also possible to identify a 
non-diagonal source covariance operator $S:L^2(\Omega)\to L^2(\Omega)$ 
describing 
correlated sources from corresponding  correlation data
\begin{equation}
\mathcal{C}(S) = \mathcal{G} S \mathcal{G}^* . \label{eq:corrfwd} 
\end{equation}
\begin{enumerate}
    \item \emph{Non-uniqueness without priori information:} We first show that 
    this is not possible without prior information on $S$.
As already mentioned in the introduction, the operator $\mathcal{G}$ has a non-trivial kernel $\textnormal{ker}\left( \mathcal{G} \right)$ (i.e.\  deterministic sources in $\Omega$ are not uniquely determined from their generated wave fields in $\mathbb{M}$).  
As covariance operators are self-adjoint and positive semi-definite, 
we can decompose $\mathcal{C}(S)$ into $\mathcal{C}(S) = L L^*$ with $L:= \mathcal{G}S^{\nicefrac{1}{2}}$. But there exist positive semi-definite operators $S^{\nicefrac{1}{2}} \neq 0$ 
such that $\textnormal{ran} \left(  S^{\nicefrac{1}{2}} \right) \subset \textnormal{ker} \left( \mathcal{G} \right)$. For such operators we have  $\mathcal{C}(S) = 0$, i.e.\ the forward operator in \eref{eq:corrfwd} is not injective. 
\item  \emph{Structured correlations:} One may ask if one can identify a source covariance operator $S$ if a certain structure of $S$ is a-priori known. E.g.\ one might 
assume that $S$ is a convolution operator or the form $S=\tilde{S}M_q\tilde{S}^*$ with some 
known operator $\tilde{S}$. However, the extension of our uniqueness result to such situations 
seems to require new ideas since our proof relies on the fact that the span of products 
$\overline{u_1}u_2$ with $u_1,u_2\in \mathrm{ran}(\mathcal{G}^*)$ is dense whereas $\mathrm{ran}(\mathcal{G}^*)$ is not dense, and it is not obvious if this result remains true if 
an operator is applied to $u_1$ or $u_2$ or both.
\end{enumerate}
\end{rem}


\section{A continuous perspective on common reconstruction methods}\label{sec:incorporation}
In this section we will analyze three common source reconstruction methods that are used for aeroacoustic measurement data, namely \textit{Covariance Matrix Fitting} (CMF) (also known as \textit{spectral estimation method}) \cite{Blacodon2004, Yardibi2008}, \textit{Conventional Beamforming} (CBF) \cite{Veen1988, Johnson1993} and \textit{DAMAS} \cite{Brooks2006}. All three methods can be generalized to the infinite dimensional framework, presented in the last two sections. For a broader overview on microphone array techniques for aeroacoustic purposes we refer to \cite{Leclre2017, Merino-Martinez2019}. 

\subsection{The adjoint of the forward operator}
Before we start with the specific source reconstruction approaches, we need to characterize the adjoint of the forward operator.
\begin{prop}[Adjoint forward operator] \label{prop:adj_fwd}
The adjoint of the forward operator 
\begin{equation*}
\mathcal{C}: L^2(\Omega) \rightarrow \HS(L^2(\mathbb{M})), \ \ \ \ \ \ q \mapsto \mathcal{C}(q)
\end{equation*}
is given by  
\begin{equation*}
\mathcal{C}^*:   \HS(L^2(\mathbb{M})) \rightarrow L^2(\Omega), \ \ \ \ \ \ \left( \mathcal{C}^*K \right)(\mathbf{y}) = \left \langle K, \mathcal{P}_{\mathbf{y}}  \right \rangle_{\HS} ,
\end{equation*}
for $K\in \HS(L^2(\mathbb{M}))$ and $\mathbf{y}\in \Omega$ 
with the monopole operator $P_{\mathbf{y}}\in \HS(L^2(\mathbb{M}))$ defined by 
\begin{equation*}
\left( \mathcal{P}_{\mathbf{y}} \varphi \right)(\mathbf{x}_1) = \int_{\mathbb{M}} g(\mathbf{x}_1, \mathbf{y}) \overline{g(\mathbf{x}_2, \mathbf{y})} \varphi(\mathbf{x}_2) d\mathbf{x}_2 
\end{equation*}
for $\varphi\in L^2(\mathbb{M})$ and $\mathbf{x}_1\in \mathbb{M}$. 
\end{prop}
\begin{proof}
See \ref{sec:appendix}
\end{proof}
Note that if $K\in \HS(L^2(\mathbb{M}))$ is given by its integral kernel $k\in L^2(\mathbb{M}\times \mathbb{M})$, i.e.\ $(K\varphi)(\mathbf{x}_1) = 
\int_{\mathbb{M}} k(\mathbf{x}_1,\mathbf{x}_2)\varphi(\mathbf{x}_1)\,d\mathbf{x}_2$, then in view of the isometry \eref{eq:HSnormCq} we have
\begin{equation}\label{eq:adjoint_kernel}
(\mathcal{C}^*K)(\mathbf{y}) 
= \int_{\mathbb{M}} \int_{\mathbb{M}} 
g(\mathbf{x}_1, \mathbf{y}) \overline{g(\mathbf{x}_2, \mathbf{y})} k(\mathbf{x}_1,\mathbf{x}_2)\,d\mathbf{x}_1,d\mathbf{x}_2\,.
\end{equation}

Recall the elements of the discrete measurement setup, presented in Section \ref{sec:fwd_prob}: \ 
microphone a positions $\left \lbrace \mathbf{x}_m \right \rbrace_{m=1}^M \subset \mathbb{M}$, 
focus points $\left \lbrace \mathbf{y}_n \right \rbrace_{n=1}^N$, 
propagation matrix $ \underline{\mathcal{G}} \in \mathbb{C}^{M \times N} \ $ \eref{eq:prop_mat} and source matrix  $ \underline{M_q} \in \mathbb{R}^{N \times N} \ $. \ \\
The discrete forward operator $\underline{\mathcal{C}}$ is thus defined as
\begin{equation*}
\underline{\mathcal{C}}: \ \mathbb{R}^N \rightarrow \mathbb{C}^{M \times M} , \ \ \ \ \underline{q} \mapsto \underline{\mathcal{G}}  \underline{M_q}  \underline{\mathcal{G}}^{*} \ .
\end{equation*}
In the following we present each reconstruction method in a discrete version for an observed covariance matrix $\underline{C}^{\mathrm{obs}} \in \mathbb{C}^{M \times M}$ and an infinite dimensional version for an observed covariance operator $C^{\mathrm{obs}} \in \HS\left( L^2(\mathbb{M}) \right)$.
\subsection{Covariance Matrix Fitting}
For infinite dimensional quantities we have the least squares problem 
\begin{equation} \label{eq:cmf}
\norm{\mathcal{C}(q) -  C^{\mathrm{obs}}}_{\HS}^2=\min! \ ,
\end{equation}
which is uniquely solvable for $q \in L^{\infty}\left( \Omega \right)$ and exact data $C^{\mathrm{obs}} \in \textnormal{ran} \left(\mathcal{C} \right)$ by Theorem \ref{thm:uniqueness}. In the discrete version, the CMF problem is defined by the 
least squares problem 
\begin{equation}\label{eq:cmf_disc}
\norm{\underline{\mathcal{C}}(\underline{q}) - \underline{C}^{\mathrm{obs}}}_F^2=\min!  
\end{equation}
where $\|\cdot\|_F$ denotes the Frobenius norm, the discrete analog of the Hilbert-Schmidt norm. In \cite{Blacodon2004}, the minimization problem \eref{eq:cmf_disc} is solved for an experimental dataset of a wind tunnel experiment with an aircraft wing.

\subsection{Conventional Beamforming}\label{sec:beamforming}
Conventional Beamforming is probably the most popular evaluation method for aeroacoustic measurement data since it yields a fast and robust estimator of the source power. Instead of solving an inverse problem for all source powers at once, CBF estimates the source power at each focus point $\mathbf{y}_n$ separately. Such methods are often referred to as \textit{array imaging methods}. For a broad overview on different imaging scenarios and their analysis we refer to \cite{Garnier2016}.  

The Beamforming imaging functional $\mathcal{I}: \  \Omega \rightarrow \mathbb{R} $ is defined as
\begin{equation*}
\mathcal{I}(\mathbf{y}) := \argmin_{\mu \in \mathbb{R}} \norm{C^{\mathrm{obs}} - \mu \mathcal{P}_{\mathbf{y}}}_{\HS}^2 \ .
\end{equation*}
If we assume that the empirical estimate $C^{\mathrm{obs}}$ 
of the covariance operator is self-adjoint, 
this one-dimensional minimization problem has the solution 
\begin{equation}\label{eq:I_beamforming}
\mathcal{I}(\mathbf{y}) = \real{ \frac{\langle C^{\mathrm{obs}}, \mathcal{P}_{\mathbf{y}} \rangle_{\HS}}{\norm{ \mathcal{P}_{\mathbf{y}}}_{\HS}^2} } =  \frac{\langle C^{\mathrm{obs}}, \mathcal{P}_{\mathbf{y}} \rangle_{\HS}}{\norm{ \mathcal{P}_{\mathbf{y}}}_{\HS}^2} = \frac{\left( \mathcal{C}^*(C^{\mathrm{obs}}) \right)(\mathbf{y})}{\norm{ \mathcal{P}_{\mathbf{y}}}_{\HS}^2} \ .
\end{equation}
Here the second equality follows from eq.~\eref{eq:adjoint_kernel}.
In helioseismology imaging functionals analogous to \eref{eq:I_beamforming} (without the scaling factor $\|P_{\mathbf{y}}\|^2$) appear as special types of holographic imaging functionals 
(see \cite{GFYBB:18,LB:00a}).

The Hilbert-Schmidt norm of the monopole operator is given by
\begin{eqnarray*}
\norm{\mathcal{P}_{\mathbf{y}}}_{\HS}^2 &= \int_{\mathbb{M} \times \mathbb{M}} \abs{g(\mathbf{x}_1, \mathbf{y}) \overline{g(\mathbf{x}_2, \mathbf{y})} }^2 d(\mathbf{x}_1, \mathbf{x}_2) = \norm{g(\cdot, \mathbf{y} )}_{L^2(\mathbb{M})}^4 > 0 \ .
\end{eqnarray*}

For discrete data and a fixed focus point $\mathbf{y}_n \in \Omega$, the \textit{steering vector} $\mathbf{g}(\mathbf{y}_n) \in \mathbb{C}^{M}$ is defined as the pointwise evaluation of the Green's function at all microphones
\begin{equation*}
\mathbf{g}(\mathbf{y}_n) =  \left( \begin{array}{ccc}
g(\mathbf{x}_1, \mathbf{y}_n) \\
\vdots  \\
g(\mathbf{x}_M, \mathbf{y}_n) \end{array} \right) \ .
\end{equation*}
The discrete monopole operator $\underline{\mathcal{P}_{\mathbf{y}_n}} = \mathbf{g}(\mathbf{y}_n) \mathbf{g}(\mathbf{y}_n)^* \in \mathbb{C}^{M \times M}$ is therefore called \textit{steering matrix}. In analogy to Proposition \ref{prop:adj_fwd}, the discrete adjoint forward operator is 
\begin{equation*}
\underline{\mathcal{C}}^*: \ \mathbb{C}^{M \times M} \rightarrow \mathbb{R}^N , \ \ \ \ \left( \underline{\mathcal{C}}^*\left(\underline{K} \right) \right)_n = \left \langle \underline{K},  \underline{\mathcal{P}_{\mathbf{y}_n}} \right \rangle_F  \ ,
\end{equation*}
where $\langle \cdot , \cdot \rangle_F$ denotes the Frobenius scalar product. Thus the discrete Beamforming functional is defined as
\begin{equation*}
\underline{\mathcal{I}}(\mathbf{y}_n) = \frac{ \mathbf{g}(\mathbf{y}_n)^* \underline{C}^{\mathrm{obs}} \mathbf{g}(\mathbf{y}_n) }{\abs{\mathbf{g}(\mathbf{y}_n)}^4} = \frac{\langle \underline{C}^{\mathrm{obs}}, \underline{\mathcal{P}_{\mathbf{y}_n}} \rangle_{F}}{\big\| \underline{\mathcal{P}_{\mathbf{y}_n}}\big\|_{F}^2} = \frac{\left( \underline{\mathcal{C}}^*(\underline{C}^{\mathrm{obs}}) \right)(\mathbf{y}_n)}{\big\| \underline{\mathcal{P}_{\mathbf{y}_n}}\big\|_{F}^2} \ .
\end{equation*}
\begin{rem}[Time domain array imaging]
CBF can be motivated by the time domain principle of \textit{delay and sum} (DAS) \cite{Johnson1993}, which is strongly related to Kirchhoff migration (see e.g.\ \cite{Borcea2005}). The basic idea of DAS is to shift all sensor time signals according to the time delay to a fixed focus point. Summing up the shifted signals, source signals originating at the focus point accumulate. 
\end{rem}

\subsection{DAMAS}
The idea of DAMAS (\textit{deconvolution approach for the mapping of acoustic sources}) is to deblur the source information obtained by a CBF solution. It is defined by an integral equation of the first kind, 
\begin{equation} \label{eq:damas_orig}
\mathcal{I}(\mathbf{y}) = \int_{\Omega}  \psi(\mathbf{y}, \mathbf{y}')  q(\mathbf{y}') d \mathbf{y}' \ . 
\end{equation}
In this article, DAMAS will always refer to the integral equation \eref{eq:damas_orig} and not to the iterative Gauß-Seidel method that was suggested in \cite{Brooks2006} in order to solve the discrete version of \eref{eq:damas_orig}. The integral kernel $\psi$ is usually referred to as point-spread function (PSF) and defined as
\begin{equation*}
\psi(\mathbf{y}, \mathbf{y}') = \frac{\langle\mathcal{P}_{\mathbf{y}}, \mathcal{P}_{\mathbf{y}'} \rangle_{\HS}}{\norm{ \mathcal{P}_{\mathbf{y}'}}_{\HS}^2} = \frac{\left( \mathcal{C}^*(\mathcal{P}_{\mathbf{y}}) \right)(\mathbf{y}')}{\norm{ \mathcal{P}_{\mathbf{y}'}}_{\HS}^2} \ .
\end{equation*}
For a shift invariant PSF, \eref{eq:damas_orig} reduces to a convolution integral, but for our scenario the PSF is not shift invariant. Nevertheless, deblurring methods like DAMAS are usually called \textit{deconvolution methods} in the aeroacoustic community. The next statement relates the integral equation of DAMAS to the least squares problem of CMF. 
\begin{prop}[Normal equation] \label{prop:normal_eq}
The DAMAS problem \eref{eq:damas_orig} is equivalent to the operator equation
\begin{equation}\label{eq:cmf_normal}
\mathcal{C}^*\mathcal{C}(q) = \mathcal{C}^*C^{\mathrm{obs}}
\end{equation}
which is the normal equation of the CMF problem \eref{eq:cmf}.
\end{prop}
\begin{proof}
First of all we can multiply \eref{eq:damas_orig} by  $\norm{\mathcal{P}_{\mathbf{y}}}_{\HS}^2$ which yields the equivalent integral equation
\begin{equation} \label{eq:damas_resc}
\left( \mathcal{C}^*(C^{\mathrm{obs}}) \right)(\mathbf{y}) = \int_{\Omega}  \langle\mathcal{P}_{\mathbf{y}}, \mathcal{P}_{\mathbf{y}'} \rangle_{\HS}  \ q(\mathbf{y}') d \mathbf{y}' \ .
\end{equation}
For an orthonormal basis $\lbrace \varphi_j  \rbrace_{j \in \mathbb{N}}$ of $L^2(\mathbb{M})$, reformulating the right-hand side of \eref{eq:damas_resc} yields
\begin{eqnarray}
\int_{\Omega}  \langle\mathcal{P}_{\mathbf{y}}, \mathcal{P}_{\mathbf{y}'} \rangle_{\HS}  \ q(\mathbf{y}') d \mathbf{y}' &=& \int_{\Omega}  \sum \limits_{j=1}^{\infty} \langle \mathcal{P}_{\mathbf{y}} \varphi_j, \mathcal{P}_{\mathbf{y}'} \varphi_j \rangle_{L^2(\mathbb{M})}  \ q(\mathbf{y}') d \mathbf{y}' \nonumber \\
&=& \sum \limits_{j=1}^{\infty} \int_{\Omega} \langle \mathcal{P}_{\mathbf{y}} \varphi_j, \mathcal{P}_{\mathbf{y}'} \varphi_j \rangle_{L^2(\mathbb{M})}  \ q(\mathbf{y}') d \mathbf{y}' \nonumber \\
&=& \sum \limits_{j=1}^{\infty} \int_{\mathbb{M}} \left[ \int_{\Omega} \left( \mathcal{P}_{\mathbf{y}'} \varphi_j \right)(\mathbf{x}) q(\mathbf{y}') d \mathbf{y}' \right] \overline{\left( \mathcal{P}_{\mathbf{y}} \varphi_j \right)(\mathbf{x})} d\mathbf{x} \nonumber \\ \ \label{eq:rhschar} \ . 
\end{eqnarray}
We obtain further
\begin{eqnarray}
\left[ \dots \right]  &=& \int_{\Omega} \left( \int_{\mathbb{M}} g(\mathbf{x},\mathbf{y}') \overline{g(\mathbf{x}', \mathbf{y}')} \varphi_j(\mathbf{x}')d\mathbf{x}' \right) q(\mathbf{y}') d \mathbf{y}' \nonumber \\ &=& \int_{\mathbb{M}} c_q(\mathbf{x}, \mathbf{x}') \varphi_j(\mathbf{x}')d\mathbf{x}' =  \left( \mathcal{C}(q)\varphi_j \right)(\mathbf{x}) \ .  \label{eq:cqchar}
\end{eqnarray}
Inserting \eref{eq:cqchar} into \eref{eq:rhschar} yields
\begin{equation*}
\int_{\Omega}  \langle\mathcal{P}_{\mathbf{y}}, \mathcal{P}_{\mathbf{y}'} \rangle_{\HS}  \ q(\mathbf{y}') d \mathbf{y}' = \left \langle \mathcal{C}(q) , \mathcal{P}_{\mathbf{y}}   \right \rangle_{\HS} =  \left( \mathcal{C}^*\mathcal{C}(q) \right)(\mathbf{y}) \ .
\qedhere
\end{equation*}
\end{proof}

\begin{cor}[DAMAS uniqueness]
For exact data $C^{\mathrm{obs}} \in \textnormal{ran}\left( \mathcal{C} \right) $ and the source space $L^{\infty}\left( \Omega \right)$ the solution of \eref{eq:damas_orig} is unique. 
\end{cor}
\begin{proof}
By the uniqueness of \eref{eq:cmf} we obtain
\[
\textnormal{ker} \left( \mathcal{C}^*\mathcal{C} \right) \cap L^{\infty}(\Omega ) =  \textnormal{ker} \left( \mathcal{C} \right) \cap L^{\infty}(\Omega ) = \lbrace 0 \rbrace \ .\qedhere
\]
\end{proof}

The original, discrete version of DAMAS is given by the linear system
\begin{equation} \label{eq:damas_disc}
\underline{\mathcal{I}}(\mathbf{y}_n) = \sum \limits_{n' =1}^N \underline{\psi}\left( \mathbf{y}_n, \mathbf{y}_{n'} \right) q(\mathbf{y}_{n'}) \ ,
\end{equation} 
with the discrete point spread function
\begin{equation*}
\underline{\psi}\left( \mathbf{y}_n, \mathbf{y}_{n'} \right) = \frac{\Big \langle \underline{\mathcal{P}_{\mathbf{y}_n}}, \underline{\mathcal{P}_{\mathbf{y}_{n'}}} \Big \rangle_F}{\big\|\underline{\mathcal{P}_{\mathbf{y}_{n'}}}\big\|_F^2} \ .
\end{equation*}
Similar to Proposition \ref{prop:normal_eq} the discrete CMF and DAMAS problem are related by the normal equation. 
\begin{cor}[Discrete normal equation]
The problem \eref{eq:damas_disc} is equivalent to the linear system 
\begin{equation}\label{eq:cmf_normal_disc}
\underline{\mathcal{C}}^*\underline{\mathcal{C}}(\underline{q}) = \underline{\mathcal{C}}^*\underline{C}^{\mathrm{obs}}
\end{equation}
which is the normal equation of \eref{eq:cmf_disc}.
\end{cor}

\subsection{Regularization}
Since the operator $\mathcal{G}$ is infinitely smoothing, the inverse problem \eref{eq:opeq_uniqueness} is ill-posed. Therefore, at least for fine discretizations of the source intensity $q$ regularization is required to obtain stable reconstructions in the presence of noise. In case of CMF this leads to estimators of the form 
\begin{equation}\label{eq:cmf_tikhonov}
\widehat{q}_{CMF,\alpha} \in \argmin_{q\in L^{\infty}(\Omega)}\left[ \norm{\mathcal{C}(q) - C^{\mathrm{obs}}}_{\footnotesize{\HS}(L^2(\mathbb{M}))}^2 + \alpha \mathcal{R}(q)\right]
\end{equation}
and for the DAMAS problem
\begin{equation}\label{eq:damas_tikhonov}
\widehat{q}_{DAMAS,\alpha} \in \argmin_{q\in L^{\infty}(\Omega)}\left[  \norm{\mathcal{C}^*\mathcal{C}(q) - \mathcal{C}^*C^{\mathrm{obs}}}_{L^2(\Omega)}^2 + \alpha \mathcal{R}(q)\right] \ .
\end{equation}
where $\mathcal{R}$ is a convex penalty term and $\alpha>0$ is a regularization parameter. In \cite{Yardibi2008}, the authors present discrete versions of \eref{eq:cmf_tikhonov} and \eref{eq:damas_tikhonov} using non-negativity constraints, box constraints on the sum of the source intensities and sparsity enforcing penalties.

\subsection{Reconstructions from experimental data}
To conclude this section we illustrate the application of the presented methods in an experimental setup. For a typical aeroacoustic experiment, a solid object (for example a model of an aircraft) is placed inside the velocity field of a wind tunnel. The fluid structure interactions generate an acoustic signal, which is measured by a microphone array. The raw time data is further processed to an estimator of the cross correlations $\underline{C}^{\mathrm{obs}}$. The reconstruction of the source powers is often called \textit{source map}. Figure  \ref{fig:sourcemaps} shows an example of a source map for each method. The results for CMF and DAMAS are obtained by quadratic Tikhonov regularization with a non-negativity constraint.  
\begin{figure}
\includegraphics[width = \textwidth]{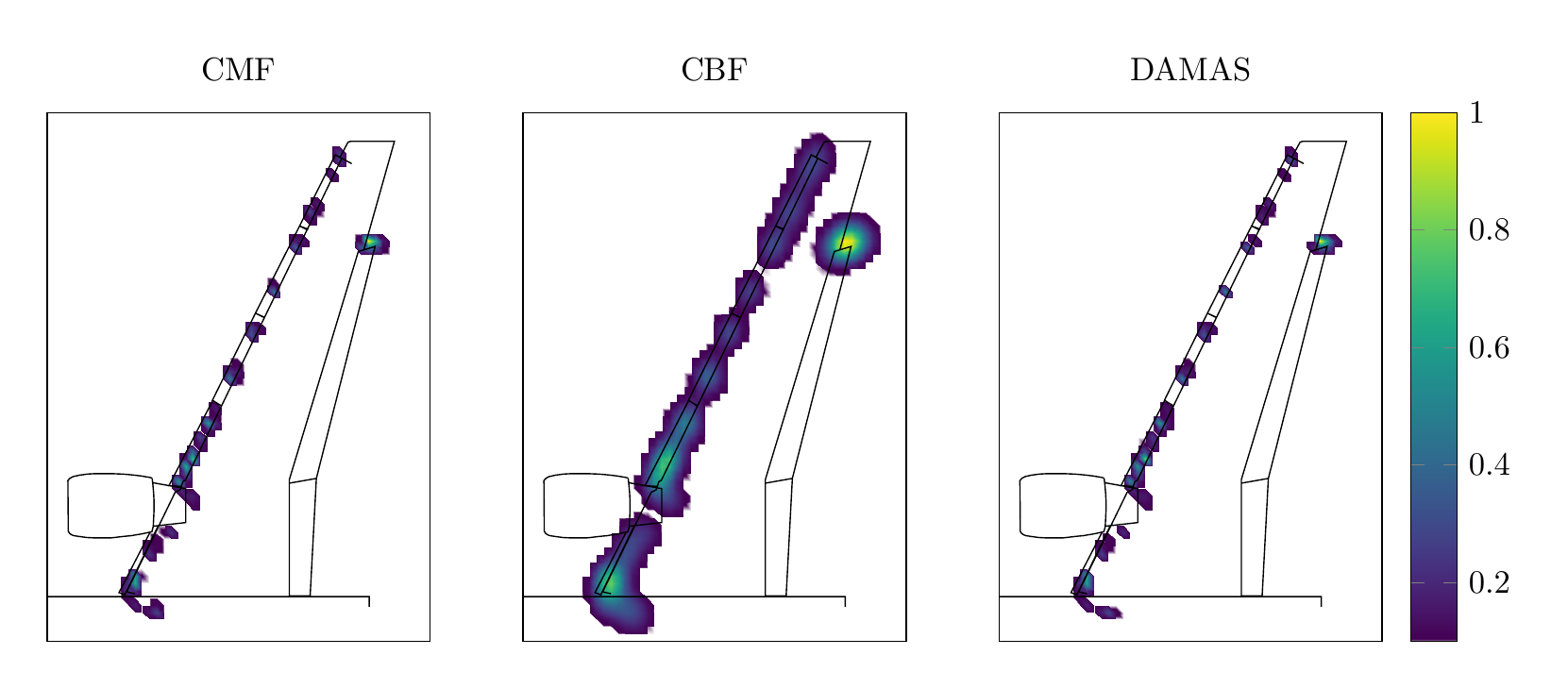}
\caption{Examples of source maps for a Dornier-728 half-model, measured at the cryogenic wind tunnel in Cologne (DNW-KKK) \cite{Ahlefeldt2013}, \cite{Bahr2017}. The source powers are shown on a cut plane through the wing cross section and normalized between $0$ and $1$. In order to cut the noise floor, values below $0.1$ are hidden. The evaluation frequency is $f=8$kHz and the Mach number $\abs{\mathbf{m}} = 0.15$.} \label{fig:sourcemaps}
\end{figure}


\section{Conclusions and outlook} 
Sound source reconstruction methods such as CBF, DAMAS or CMF  are all based on the same discrete sound propagation model. Here we described these methods in an infinite dimensional setup which allowed to relate the reconstruction methods to each other via the adjoint forward operator and the normal equation. 

Of course the free field propagation model used in this paper relies on strong simplifications such as the negligence of solid geometries (aeroacoustic model, wind tunnel walls) which are present in aeroacoustic experiments.
Moreover, we have assumed completely uncorrelated sources, and one may ask if this assumption can be relaxed to some types of structured source correlations. Nevertheless, the discrete version of this simplified forward operator has been successfully applied in aeroacoustic testing for many decades and is still state-of-the-art for most applications in this field. 
The validation of the results for experimental data remains challenging since a ground truth for the source power function is usually not known in such cases. 

As a main result we proved the injectivity of the forward operator for bounded sources. This gives rise to a number of further questions which may be addressed in future research. 
First, one may try to extend the technique of our uniqueness proof to more complicated geometries such as waveguides or the presence of known obstacles or inhomogeneous background media. A second natural direction of research concerns the extension of our uniqueness proof to a conditional stability result or even variational source conditions, which will most likely be of logarithmic type under natural smoothness assumptions. The latter would even yield error bounds for reconstruction methods such as regularized CMF, i.e.\ Tikhonov regularization (see \cite{HohageWeidling2015}). 
Finally, the results of this paper may eventually lead to a better theoretical understanding and justification of helioseismic holography.

\section*{Acknowledgments} 
We would like to thank Dan Yang for helpful discussions on helioseismic holography.

\appendix

\section{Proofs of auxiliary statements}\label{sec:appendix}

\begin{proof}[Proof of Proposition \ref{prop:HS_cov} (Mapping property of the forward operator)]
In view of \eref{eq:HSnormCq} we have to 
show that $c_q \in L^2 \left( \mathbb{M} \times \mathbb{M} \right)$. 
Since $\mathbb{M}$ and $\Omega$ are disjoint and bounded, the integral kernel 
\begin{equation*}
    \kappa(\mathbf{x}_1, \mathbf{x}_2, \mathbf{y}) = g(\mathbf{x}_1, \mathbf{y}) \overline{g(\mathbf{x}_2, \mathbf{y})}
\end{equation*}
is continuous on $\mathbb{M} \times \mathbb{M} \times \Omega$ and therefore $\kappa \in L^2 \left( \mathbb{M} \times \mathbb{M} \times \Omega \right)$. This implies
\begin{eqnarray*}
\norm{c_q}_{L^2(\mathbb{M} \times \mathbb{M})}^2 &=& \int_{\mathbb{M} \times \mathbb{M}} \abs{ \int_{\Omega} \kappa(\mathbf{x}_1, \mathbf{x}_2, \mathbf{y}) q(\mathbf{y})   \ d\mathbf{y} }^2 \ d(\mathbf{x}_1, \mathbf{x}_2) \\
&\leq & \norm{q}_{L^2(\Omega)}^2 \int_{\mathbb{M} \times \mathbb{M}} \int_{\Omega} \abs{\kappa(\mathbf{x}_1, \mathbf{x}_2, \mathbf{y})}^2   \ d\mathbf{y} \ d(\mathbf{x}_1, \mathbf{x}_2) \\
&=&  \norm{q}_{L^2(\Omega)}^2  \norm{\kappa}_{L^2 ( \mathbb{M} \times \mathbb{M} \times \Omega)}^2 \ .
\end{eqnarray*}
\end{proof}

\begin{proof}[Proof of Proposition \ref{prop:adj_fwd} (Adjoint Forward Operator)]
For $q \in L^2(\Omega)$ and $K \in  \HS(L^2(\mathbb{M}))$, we need to show that 
\begin{equation} \label{eq:adjoint_prop}
\int_{\Omega} q(\mathbf{y}) \overline{\langle K, \mathcal{P}_{\mathbf{y}} \rangle_{\HS}} \,d\mathbf{y}  
= \left \langle \mathcal{C} q ,  K \right \rangle_{\HS} \ .
\end{equation}
Note that  $\mathcal{P}_{\mathbf{y}}$ is in fact a  Hilbert-Schmidt operator on $L^2(\mathbb{M})$ with integral kernel $l_{\mathbf{y}}(\mathbf{x}_1, \mathbf{x}_2) =  g(\mathbf{x}_1, \mathbf{y}) \overline{g(\mathbf{x}_2, \mathbf{y})} $, that $C:=\sup_{\mathbf{y}\in\Omega}\|\mathcal{P}_{\mathbf{y}}\|_{\HS} = \sup_{\mathbf{y}\in\Omega}\|l_{\mathbf{y}}\|_{L^2(\Omega\times \Omega)}$ is finite, and that both $l_{\mathbf{y}}$ and 
$\mathcal{P}_{\mathbf{y}}$ depend continuously on 
$\mathbf{y}$ with respect to the natural norms. 
Let $\lbrace \varphi_j \rbrace$ be an orthonormal basis of $L^2(\mathbb{M})$. Then $\overline{\langle K,P_{\mathbf{y}}\rangle_{\HS}}= \tr (K^*P_{\mathbf{y}}) = \sum_{j=1}^\infty \langle P_{\mathbf{y}}\varphi_j, K\varphi_j\rangle$.
The sequence
\begin{equation*}
f_n(\mathbf{y}) = \sum \limits_{j=1}^n \langle  \mathcal{P}_{\mathbf{y}} \varphi_j, K \varphi_j \rangle_{L^2(\mathbb{M})}
\end{equation*}
converges pointwise for all $\mathbf{y} \in \Omega$ and has the integrable majorant $C\|K\|_{\HS}$. Thus we can interchange the integration over $\Omega$ and the infinite sum in \eref{eq:adjoint_prop} by the dominated convergence theorem. This yields
\begin{eqnarray*}
\int_{\Omega} q(\mathbf{y})  \overline{\langle K, \mathcal{P}_{\mathbf{y}} \rangle_{\HS}} \,d\mathbf{y} 
&=& \sum \limits_{j=1}^{\infty} \int_{\Omega} q(\mathbf{y}) \langle  \mathcal{P}_{\mathbf{y}} \varphi_j, K \varphi_j \rangle_{L^2(\mathbb{M})} \,d \mathbf{y} \\
&=& \sum \limits_{j=1}^{\infty} \int_{\Omega} q(\mathbf{y})  \int_{\mathbb{M}}\left(\mathcal{P}_{\mathbf{y}} \varphi_j\right) (\mathbf{x}_1) \overline{\left( K \varphi_j\right)(\mathbf{x}_1)} d \mathbf{x}_1 
d \mathbf{y} \\
&=& \sum \limits_{j=1}^{\infty} \int_{\mathbb{M}}  
\overline{\left( K \varphi_j\right)(\mathbf{x}_1)}
 \int_{\Omega} q(\mathbf{y})  \left(\mathcal{P}_{\mathbf{y}} \varphi_j\right) (\mathbf{x}_1) \,d\mathbf{y}  \,d\mathbf{x}_1 \\
 &=& \sum \limits_{j=1}^{\infty} \int_{\mathbb{M}} 
\overline{\left( K \varphi_j\right)(\mathbf{x}_1)}
 \left(\mathcal{C}(q) \varphi_j \right)(\mathbf{x}_1) \, d \mathbf{x}_1 = \left \langle \mathcal{C}q, K \right \rangle_{\HS} 
\end{eqnarray*}
where we have used the Fubini-Tonelli theorem. 
\end{proof}

\section*{References}
\bibliography{references}

\begin{thebibliography}{10}

\bibitem{Ahlefeldt2013}
T.~Ahlefeldt.
\newblock Aeroacoustic measurements of a scaled half-model at high reynolds
  numbers.
\newblock {\em {AIAA} Journal}, 51(12):2783--2791, December 2013.

\bibitem{Alexandridis2015}
A.~Alexandridis and A.~Mouchtaris.
\newblock Multiple sound source location estimation and counting in a wireless
  acoustic sensor network.
\newblock In {\em 2015 {IEEE} Workshop on Applications of Signal Processing to
  Audio and Acoustics ({WASPAA})}. {IEEE}, oct 2015.

\bibitem{Bahr2017}
C.~J. Bahr, W.~M. Humphreys, D.~Ernst, T.~Ahlefeldt, C.~Spehr, A.~Pereira,
  Q.~Lecl{\`{e}}re, C.~Picard, R.~Porteous, D.~Moreau, J.~R. Fischer, and C.~J.
  Doolan.
\newblock A comparison of microphone phased array methods applied to the study
  of airframe noise in wind tunnel testing.
\newblock In {\em 23\textsuperscript{rd} {AIAA}/{CEAS} Aeroacoustics
  Conference}. American Institute of Aeronautics and Astronautics, June 2017.

\bibitem{Barsikow1987}
B.~Barsikow, W.F. King, and E.~Pfizenmaier.
\newblock Wheel/rail noise generated by a high-speed train investigated with a
  line array of microphones.
\newblock {\em Journal of Sound and Vibration}, 118(1):99--122, oct 1987.

\bibitem{Billingsley1976}
J.~Billingsley and R.~Kinns.
\newblock The acoustic telescope.
\newblock {\em Journal of Sound and Vibration}, 48(4):485--510, October 1976.

\bibitem{Blacodon2004}
D.~Blacodon and G.~Elias.
\newblock Level estimation of extended acoustic sources using a parametric
  method.
\newblock {\em Journal of Aircraft}, 41(6):1360--1369, November 2004.

\bibitem{Bleistein2001}
N.~Bleistein, J.~W. Stockwell, and J.~K. Cohen.
\newblock {\em Mathematics of Multidimensional Seismic Imaging, Migration, and
  Inversion}.
\newblock Springer New York, 2001.

\bibitem{Borcea2005}
L.~Borcea, G.~Papanicolaou, and C.~Tsogka.
\newblock Interferometric array imaging in clutter.
\newblock {\em Inverse Problems}, 21(4):1419--1460, July 2005.

\bibitem{Brooks2006}
T.~F. Brooks and W.~M. Humphreys.
\newblock A deconvolution approach for the mapping of acoustic sources
  ({DAMAS}) determined from phased microphone arrays.
\newblock {\em Journal of Sound and Vibration}, 294(4-5):856--879, 2006.

\bibitem{Cakoni2006}
F.~Cakoni and D.~Colton.
\newblock {\em Qualitative Methods in Inverse Scattering Theory}.
\newblock Springer-Verlag, 2006.

\bibitem{Chapman2000}
C.J. Chapman.
\newblock Similarity variables for sound radiation in a uniform flow.
\newblock {\em Journal of Sound and Vibration}, 233(1):157--164, May 2000.

\bibitem{Colton2013}
D.~Colton and R.~Kress.
\newblock {\em Inverse Acoustic and Electromagnetic Scattering Theory}.
\newblock Springer New York, 3rd edition, 2013.

\bibitem{Farmani2017}
M.~Farmani, M.~Syskind Pedersen, Z.-H. Tan, and J.~Jensen.
\newblock Informed sound source localization using relative transfer functions
  for hearing aid applications.
\newblock {\em {IEEE}/{ACM} Transactions on Audio, Speech, and Language
  Processing}, 25(3):611--623, mar 2017.

\bibitem{Ganguly2017}
A.~Ganguly and I.~Panahi.
\newblock Non-uniform microphone arrays for robust speech source localization
  for smartphone-assisted hearing aid devices.
\newblock {\em Journal of Signal Processing Systems}, 90(10):1415--1435, nov
  2017.

\bibitem{Garnier2016}
J.~Garnier and G.~Papanicolaou.
\newblock {\em Passive Imaging with Ambient Noise}.
\newblock Cambridge University Press, 2016.

\bibitem{GFYBB:18}
L.~Gizon, D.~Fournier, D.~Yang, A.~C. Birch, and H.~Barucq.
\newblock Signal and noise in helioseismic holography.
\newblock {\em A\&A}, 620:A136, 2018.

\bibitem{Gombots2018}
S.~Gombots, M.~Kaltenbacher, and B.~Kaltenbacher.
\newblock Inverse scheme for acoustic source localization in 3d.
\newblock In {\em EURONOISE2018}, 2018.

\bibitem{Gregory2015}
A.~L Gregory, S.~Sinayoko, A.~Agarwal, and J.~Lasenby.
\newblock An acoustic space-time and the lorentz transformation in
  aeroacoustics.
\newblock {\em International Journal of Aeroacoustics}, 14(7):977--1003, nov
  2015.

\bibitem{GS:16}
R.~Griesmaier and J.~Sylvester.
\newblock Far field splitting by iteratively reweighted $\ell^1$ minimization.
\newblock {\em SIAM Journal on Applied Mathematics}, 76(2):705--730, 2016.

\bibitem{GS:17}
R.~Griesmaier and J.~Sylvester.
\newblock Uncertainty principles for inverse source problems, far field
  splitting, and data completion.
\newblock {\em SIAM Journal on Applied Mathematics}, 77(1):154--180, 2017.

\bibitem{GS:17b}
R.~Griesmaier and J.~Sylvester.
\newblock Uncertainty principles for three-dimensional inverse source problems.
\newblock {\em SIAM Journal on Applied Mathematics}, 77(6):2066--2092, 2017.

\bibitem{Grosche1975}
F.-R. Grosche, J.H. Jones, and G.A. Wilhold.
\newblock Measurements of the distribution of sound source intensities in
  turbulent jets.
\newblock In {\em Aeroacoustics: Jet and Combustion Noise; Duct Acoustics},
  pages 79--92. American Institute of Aeronautics and Astronautics, jan 1975.

\bibitem{Haykin1993}
S.~Haykin, J.~Litva, and T.~J. Shepherd, editors.
\newblock {\em Radar Array Processing}.
\newblock Springer Berlin Heidelberg, 1993.

\bibitem{HohageWeidling2015}
T.~Hohage and F.~Weidling.
\newblock Verification of a variational source condition for acoustic inverse
  medium scattering problems.
\newblock {\em Inverse Problems}, 31(7):075006, jun 2015.

\bibitem{Howell1986}
G.P. Howell, A.J. Bradley, M.A. McCormick, and J.D. Brown.
\newblock De-dopplerization and acoustic imaging of aircraft flyover noise
  measurements.
\newblock {\em Journal of Sound and Vibration}, 105(1):151--167, feb 1986.

\bibitem{John1982}
F.~John.
\newblock {\em Partial Differential Equations}.
\newblock Springer New York, 4th edition, 1982.

\bibitem{Johnson1993}
D.~H. Johnson.
\newblock {\em Array Signal Processing}.
\newblock Pearson Education, 1993.

\bibitem{Kaltenbacher2018}
M.~Kaltenbacher, B.~Kaltenbacher, and S.~Gombots.
\newblock Inverse scheme for acoustic source localization using microphone
  measurements and finite element simulations.
\newblock {\em Acta Acustica united with Acustica}, 104(4):647--656, jul 2018.

\bibitem{Leclre2017}
Q.~Lecl{\`{e}}re, A.~Pereira, C.~Bailly, J.~Antoni, and C.~Picard.
\newblock A unified formalism for acoustic imaging based on microphone array
  measurements.
\newblock {\em International Journal of Aeroacoustics}, 16(4-5):431--456, jul
  2017.

\bibitem{LB:00}
C.~Lindsey and C.D. Braun.
\newblock Seismic images of the far side of the sun.
\newblock {\em Science}, 287(5459):1799--1801, 2000.

\bibitem{LB:00a}
C.~Lindsey and D.~C. {Braun}.
\newblock Basic principles of solar acoustic holography - (invited review).
\newblock {\em Solar Physics}, 192:261--284, 2000.

\bibitem{Long2019}
T.~Long, J.~Chen, G.~Huang, J.~Benesty, and I.~Cohen.
\newblock Acoustic source localization based on geometric projection in
  reverberant and noisy environments.
\newblock {\em {IEEE} Journal of Selected Topics in Signal Processing},
  13(1):143--155, mar 2019.

\bibitem{Merino-Martinez2019}
R.~Merino-Mart{\'{\i}}nez, P.~Sijtsma, M.~Snellen, T.~Ahlefeldt, J.~Antoni,
  C.~J. Bahr, D.~Blacodon, D.~Ernst, A.~Finez, S.~Funke, T.~F. Geyer,
  S.~Haxter, G.~Herold, X.~Huang, W.~M. Humphreys, Q.~Lecl{\`{e}}re,
  A.~Malgoezar, U.~Michel, T.~Padois, A.~Pereira, C.~Picard, E.~Sarradj,
  H.~Siller, D.~G. Simons, and C.~Spehr.
\newblock A review of acoustic imaging methods using phased microphone arrays.
\newblock {\em {CEAS} Aeronautical Journal}, mar 2019.

\bibitem{Moscoso2018}
M.~Moscoso, A.~Novikov, G.~Papanicolaou, and C.~Tsogka.
\newblock Robust multifrequency imaging with {MUSIC}.
\newblock {\em Inverse Problems}, 35(1):015007, dec 2018.

\bibitem{Mosher1984}
M.~Mosher.
\newblock The influence of a wind tunnel on helicopter rotational noise:
  Formulation of analysis.
\newblock Technical report, NASA Technical Memorandum 85982, 1984.

\bibitem{Allen2002}
T.~J. Mueller.
\newblock {\em Aeroacoustic Measurements}.
\newblock Springer Berlin Heidelberg, 2002.

\bibitem{Oerlemans2007}
S.~Oerlemans, P.~Sijtsma, and B.~M{\'{e}}ndez L{\'{o}}pez.
\newblock Location and quantification of noise sources on a wind turbine.
\newblock {\em Journal of Sound and Vibration}, 299(4-5):869--883, feb 2007.

\bibitem{Reed1972}
M.~Reed and B.~Simon.
\newblock {\em Methods of Modern Mathematical Physics. I: Functional Analysis}.
\newblock Academic Press, 1972.

\bibitem{Reed1975}
M.~Reed and B.~Simon.
\newblock {\em Methods of Modern Mathematical Physics. II: Fourier Analysis,
  Self-Adjointness}.
\newblock Academic Press, 1975.

\bibitem{Sijtsma2007}
P.~Sijtsma.
\newblock {CLEAN} based on spatial source coherence.
\newblock {\em International Journal of Aeroacoustics}, 6(4):357--374, dec
  2007.

\bibitem{Suzuki2010}
T.~Suzuki.
\newblock A review of diagnostic studies on jet-noise sources and generation
  mechanisms of subsonically convecting jets.
\newblock {\em Fluid Dynamics Research}, 42(1):014001, jan 2010.

\bibitem{Tachioka2014}
Y.~Tachioka, T.~Narita, S.~Watanabe, and J.~Le Roux.
\newblock Ensemble integration of calibrated speaker localization and
  statistical speech detection in domestic environments.
\newblock In {\em 2014 4th Joint Workshop on Hands-free Speech Communication
  and Microphone Arrays ({HSCMA})}. {IEEE}, may 2014.

\bibitem{Underbrink2002}
J.~R. Underbrink.
\newblock Aeroacoustic phased array testing in low speed wind tunnels.
\newblock In {\em Aeroacoustic Measurements}, pages 98--217. Springer Berlin
  Heidelberg, 2002.

\bibitem{Veen1988}
B.D.~Van Veen and K.M. Buckley.
\newblock Beamforming: a versatile approach to spatial filtering.
\newblock {\em {IEEE} {ASSP} Magazine}, 5(2):4--24, apr 1988.

\bibitem{Welch1967}
P.~Welch.
\newblock The use of fast fourier transform for the estimation of power
  spectra: A method based on time averaging over short, modified periodograms.
\newblock {\em {IEEE} Transactions on Audio and Electroacoustics},
  15(2):70--73, jun 1967.

\bibitem{WTdB:09}
J.~Wind, E.~Tijs, and H.-E. de~Bree.
\newblock Source localization using acoustic vector sensors: A {MUSIC}
  approach.
\newblock In {\em Noise and Vibration: Emerging Methods}, pages 1--10. ISVR,
  2009.

\bibitem{Yardibi2008}
T.~Yardibi, J.~Li, P.~Stoica, and L.~N. Cattafesta.
\newblock Sparsity constrained deconvolution approaches for acoustic source
  mapping.
\newblock {\em The Journal of the Acoustical Society of America},
  123(5):2631--2642, may 2008.

\end{thebibliography}
\bibliographystyle{plain}

\end{document}